\documentclass[12pt]{article}

\usepackage{amsthm,amsmath,amssymb}
\usepackage{enumerate}
\usepackage[colorlinks=true,citecolor=black,linkcolor=black,urlcolor=blue]{hyperref}
\usepackage{dsfont}
\usepackage{bm}
\usepackage[colorlinks=true,citecolor=black,linkcolor=black,urlcolor=blue]{hyperref}
\usepackage{graphicx}
\usepackage{subfigure}
\usepackage{tikz}
\newcommand{\qbc}[2]{\left[\genfrac{}{}{0pt}{}{#1}{#2}\right]}

\numberwithin{equation}{section}

\theoremstyle{plain}
\newtheorem{theorem}{Theorem}[section]
\newtheorem{lemma}[theorem]{Lemma}
\newtheorem{corollary}[theorem]{Corollary}

\newtheorem{example}[theorem]{Example}

\title{\bf Enumeration of row-increasing tableaux of two-row skew shapes}

\author{Xiaomei Chen\\
\small School of Mathematics and Computational Science\\[-0.8ex]
\small Hunan University of Science and
Technology\\[-0.8ex]
\small Xiangtan 411201, China\\
\small\tt xmchen@hnust.edu.cn\\
}

\date{
}

\begin{document}

\maketitle

\begin{abstract}
In this paper, we firstly extend a result of Bonin, Shapiro and Simion by giving the distribution of the major index over generalized Schr\"{o}der paths. Then by providing a bijection between generalized Schr\"{o}der paths and row-increasing tableaux of skew shapes with two rows, we obtain the distribution of the major index and the amajor index over these tableaux, which extends a result of Du, Fan and Zhao. We also generalize a result of Pechenik and give the distribution of the major index over increasing tableaux of skew shapes with two rows. Especially, a bijection from row-increasing tableaux with shape $(n,m)$ and maximal value $n+m-k$ to standard Young tableaux with shape $((n-k+1,m-k+1,1^k)/(1^2))$ is obtained.

\end{abstract}

\section{Introduction}
A \emph{generalized Schr\"{o}der path} is a lattice path with steps (1,0), (1,1) and (0,1) that never goes above the diagonal line $y=x$. We use $\mathrm{Sch}_k(r;n,m)$ to denote the set of generalized Schr\"{o}der paths from $(r,0)$ to $(n,m)$ with $k$ (1,1) steps. We will also denote $\mathrm{Sch}_0(r;n,m)$ as $\mathrm{Cat}(r;n,m)$, which is the set of all \emph{generalized Catalan paths} from $(r,0)$ to $(n,m)$.

In the following, we use $E$(East), $D$(Diagonal) and $N$(North) to denote the three steps (1,0), (1,1) and (0,1) respectively. In this way, we can represent a generalized Schr\"{o}der path as a word over the alphabet set $\{E, D, N\}$, and define its major index as follows. Given a word $P=p_1p_2\cdots p_{\ell}$ which is a permutation of a multiset whose elements are totally ordered, we say that $i$ is a \emph{descent} of $P$ if $p_i>p_{i+1}$. The \emph{descent set} $D(P)$ is the collection of all descents of $P$. The \emph{major index} of $P$ is defined by $\mathrm{maj}(P):=\sum_{i\in D(P)}i$.

Our first result gives the distribution of the major index over generalized Schr\"{o}der paths, which generalizes a result of Bonin, Shapiro and Simion \cite[Thm. 4.3]{Bonin-Shapiro-Simion}.
\begin{theorem}\label{thm:schroder}
Let $r,n,m$ and $k$ be positive integers with $\mathrm{Sch}_k(r;n,m)\neq \emptyset$.
If $E>N$, then we have
\begin{equation*}
\begin{aligned}
&\sum_{P\in \mathrm{Sch}_k(r;n,m)}q^{\mathrm{maj}(P)}\\
&=\qbc{n+m-r-k}{k}\bigg\{\qbc{n+m-r-2k}{n-r-k}-\qbc{n+m-r-2k}{n-k+1}\bigg\};
\end{aligned}
\end{equation*}
if $E<N$, then we have
\begin{equation*}
\begin{aligned}
&\sum_{P\in \mathrm{Sch}_k(r;n,m)}q^{\mathrm{maj}(P)}\\
&=\qbc{n+m-r-k}{k}\bigg\{\qbc{n+m-r-2k}{n-r-k}-q^{r+1}\qbc{n+m-r-2k}{n-k+1}\bigg\}.
\end{aligned}
\end{equation*}
\end{theorem}
It is obvious that the number of generalized Schr\"{o}der paths from $(n_1,m_1)$ to $(n_2,m_2)$ is equal to the number of those from $(n_1-m_1,0)$ to $(n_2-m_1,m_2-m_1)$. Therefore by setting $q=1$ in the above result, we obtain a result of Krattenthaler \cite[Thm. 10.8.1]{Krattenthaler1} which counts the number of generalized Schr\"{o}der paths between any two given lattice points.

 Generalized Schr\"{o}der paths are closely related to row-increasing tableaux. Here a \emph{row-increasing tableau} is defined as the transpose of a semistandard Young tableau such that the set of its entries is an initial segment of $\mathds{Z}_{>0}$, and an \emph{increasing tableau} is a row-increasing tableau with its columns strictly increasing. (Note that the present usage of row-increasing tableau follows the definition in \cite{Du-Fan-Zhao}, which conflicts with earlier usage in \cite{Buontempo}.) For partitions $\lambda$ and $\mu$ with $\lambda \supseteq \mu$, we denote by $\mathrm{RInc}_k(\lambda/\mu)$ (resp. $\mathrm{Inc}_k(\lambda/\mu)$) the set of row-increasing (resp. increasing) tableaux with shape $\lambda/\mu$ and maximal value $|\lambda/\mu|-k$, and denote by $\mathrm{SYT}(\lambda/\mu)$ the set of standard Young tableaux of shape $\lambda/\mu$. Then we have $\mathrm{SYT}(\lambda/\mu)=\mathrm{RInc}_0(\lambda/\mu)=\mathrm{Inc}_0(\lambda/\mu)$. Figure \ref{fig:RIncandInc} shows a row-increasing tableau $T_1\in \mathrm{RInc}_2((4,3)/(1))$ and an increasing tableau $T_2\in\mathrm{Inc}_2((4,3)/(1))$.
\begin{figure}[ht]
\begin{minipage}{0.48\linewidth}
\centering
\begin{tikzpicture}[scale=0.6]
\draw (0,1.5) node{$T_1:$};
\foreach \x in {2,...,4}
 \foreach \y in {1}
 {
  \draw (\x,\y) +(-.5,-.5) rectangle ++(.5,.5);
 }
 \foreach \x in {3,...,5}
 \foreach \y in {2}
 {
  \draw (\x,\y) +(-.5,-.5) rectangle ++(.5,.5);
 }
 \draw[red] (3,2) node{2};
 \draw (4,2) node{3};
 \draw (5,2) node{4};
 \draw[blue] (2,1) node{1};
 \draw[blue] (3,1) node{2};
 \draw[blue] (4,1) node{3};
\end{tikzpicture}
\end{minipage}
\begin{minipage}{0.48\linewidth}
\centering
\begin{tikzpicture}[scale=0.6]
\draw (0,1.5) node{$T_2:$};
\foreach \x in {2,...,4}
 \foreach \y in {1}
 {
  \draw (\x,\y) +(-.5,-.5) rectangle ++(.5,.5);
 }
 \foreach \x in {3,...,5}
 \foreach \y in {2}
 {
  \draw (\x,\y) +(-.5,-.5) rectangle ++(.5,.5);
 }
 \draw (3,2) node{1};
 \draw[red] (4,2) node{2};
 \draw[red] (5,2) node{3};
 \draw[blue] (2,1) node{1};
 \draw (3,1) node{3};
 \draw (4,1) node{4};
\end{tikzpicture}
\end{minipage}
\caption{A row-increasing tableau $T_1\in \mathrm{RInc}_2((4,3)/(1))$, and an increasing tableau $T_2\in\mathrm{Inc}_2((4,3)/(1))$, where descents of $T_i$ are colored red, and ascents of $T_i$ are colored blue.}
\label{fig:RIncandInc} 
\end{figure}

A \emph{descent} of a tableau $T$ is any instance of $i$ followed by an $i+1$ in a lower row of $T$, and we define the \emph{descent set} $D(T)$ to be the set of all descents of $T$. The \emph{major index} of $T$ is defined by $\mathrm{maj}(T)=\sum_{i\in D(T)} i$. Similarly, an \emph{ascent} of a tableau $T$ is defined in \cite{Du-Fan-Zhao} as any instance of $i$ followed by an $i+1$ in a higher row of $T$, and the \emph{amajor index} $\mathrm{amaj}(T)$ is the sum of all ascents of $T$.

Du, Fan and Zhao \cite[Thm. 3 and 4]{Du-Fan-Zhao} gave the distribution of the major and amajor index over $\mathrm{RInc}_k((n,n))$. Moreover, they noted that their results just differ from those of Bonin, Shapiro and
Simion \cite[Thm. 4.3]{Bonin-Shapiro-Simion} by the factor $q^{k(k-1)/2}$, and asked whether there is some simple explanation on these relations.

Motivated by the results of Du, Fan and Zhao, we introduce the concept of \emph{diagonal-reverse labelling} for generalized Schr\"{o}der paths. By constructing a bijection between $\mathrm{RInc}_k((n,m)/(r))$ and $\mathrm{Sch}_k(r;n,m)$, we obtain the distribution of the major and amajor index over $\mathrm{RInc}_k((n,m)/(r))$, which answers the question of Du, Fan and Zhao in a more generalized form.
\begin{theorem}\label{thm:RInc}
For positive integers $r,n,m$ and $k$ with $\mathrm{RInc}_k((n,m)/(r))\neq \emptyset$, we have
\begin{small}
\begin{equation*}
\begin{aligned}
&\sum_{T\in \mathrm{RInc}_k((n,m)/(r))}q^{\mathrm{maj}(T)}\\
&=q^{\frac{k(k-1)}{2}}\qbc{n+m-r-k}{k}\bigg\{\qbc{n+m-r-2k}{n-r-k}-\qbc{n+m-r-2k}{n-k+1}\bigg\},
\end{aligned}
\end{equation*}
\end{small}
and
\begin{small}
\begin{equation*}
\begin{aligned}
&\sum_{T\in \mathrm{RInc}_k((n,m)/(r))}q^{\mathrm{amaj}(T)}\\
&=q^{\frac{k(k-1)}{2}}\qbc{n+m-r-k}{k}\bigg\{\qbc{n+m-r-2k}{n-r-k}-q^{r+1}\qbc{n+m-r-2k}{n-k+1}\bigg\}.
\end{aligned}
\end{equation*}
\end{small}
\end{theorem}

With a bijection between $\mathrm{Inc}_k((n,n))$ and $\mathrm{SYT}((n-k,n-k,1^k))$ that preserves the major index,  Pechenik \cite[Thm. 1.1]{Pechenik} obtained the distribution of the major index over $\mathrm{Inc}_k((n,n))$ by applying the well-known $q-$hook length formula \cite[Cor. 7.21.5]{Stanley}:
\begin{equation}\label{equ:HLF}
\sum_{T\in{\mathrm{SYT}(\lambda)}}q^{\mathrm{maj}(T)}=\frac{q^{b(\lambda)}[n]!}{\prod_{u\in \lambda}[h(u)]}.
\end{equation}

 By extending the above bijection to skew shapes, we generalize Pechenik's results to $\mathrm{Inc}_k((n,m)/(r))$ as follows.
\begin{theorem}\label{thm:Inc}
For positive integers $r,n,m$ and $k$ with $\mathrm{Inc}_k((n,m)/(r))\neq \emptyset$, we have
\begin{small}
\begin{equation*}
\begin{aligned}
\sum_{T\in\mathrm{Inc}_k((n,m)/(r))}&q^{\mathrm{maj}(T)}=q^{\frac{k(k-1)}{2}}\qbc{n+m-k-r}{k}\\
&\bm{\cdot}\bigg(\qbc{n+m-2k-r}{m-k}-\frac{q^n[k]+[n][m-r]}{[n][n+1]}\qbc{n+m-2k-r}{n-k}\bigg).
\end{aligned}
\end{equation*}
\end{small}
\end{theorem}

Motivated by Pechenik's bijection, we also study connections between row-increasing tableaux and standard Young tableaux, and we obtain the following result.
\begin{theorem}\label{thm:RinctoSYT}
There is a bijection between $\mathrm{RInc}_k((n,m))$ and $\mathrm{SYT}((n-k+1,m-k+1,1^k)/(1^2))$ that preserves the major index.
\end{theorem}

This paper is organized as follows. In Section 2, we give the proof of Theorem 1.1. In section 3, we study the connections between generalized
Schr\"{o}der paths and row-increasing tableaux according to the major index, and give the proof of Theorem 1.2. In Section 4, we give a bijection between $\mathrm{Inc}_k((n,m)/(r))$ and the union of $\mathrm{SYT}((n-k,m-k,1^k)/(r))$ and $\mathrm{SYT}((n-k,m-k+1,1^{k-1})/(r))$, obtaining Theorem 1.3 as a corollary. We also give a bijection between $\mathrm{RInc}_k((n,m))$ and the union of $\mathrm{Inc}_k((n,m))$ and $\mathrm{Inc}_{k-1}((n,m-1))$, which yields the proof of Theorem 1.4.
\section{The major index for \boldmath \texorpdfstring{$\mathrm{Sch}_k(r;n,m)$}{LG}}
About the $q-$binomial coefficients, it is well known that
\begin{equation}
\label{equ:biorec}
\qbc{n}{k}=\qbc{n-1}{k}+q^{n-k}\qbc{n-1}{k-1},
\end{equation}
\begin{equation}\label{equ:q-hockey}
\sum_{s=k}^n q^{s-k}\qbc{s}{k}=\qbc{n+1}{k+1}.
\end{equation}

In \cite[Thm. 1]{Krattenthaler}, Krattenthaler and Mohanty gave a formula for counting generalized Catalan paths by major and descents, which implies the following result as a special case.
\begin{theorem}
\label{thm:catlan}
Let $r$, $n$ and $m$ be positive integers with $r\leq n$ and $m\leq n$. If $E>N$, then we have
\begin{equation*}
\label{equ:cat}
\sum_{P\in \mathrm{Cat}(r;n,m)}q^{\mathrm{maj}(P)}=\qbc{n+m-r}{n-r}-\qbc{n+m-r}{n+1};
\end{equation*}
if $E<N$, then we have
\begin{equation*}
\label{equ:cataamaj1}
\sum_{P\in \mathrm{Cat}(r;n,m)}q^{\mathrm{maj}(P)}=\qbc{n+m-r}{n-r}-q^{r+1}\qbc{n+m-r}{n+1}.
\end{equation*}
\end{theorem}

Note that there is a natural bijection between $\mathrm{RInc}_k((n,m)/(r))$ and $\mathrm{Sch}_k(r;n,m)$: given $T\in \mathrm{RInc}_k((n,m)/(r))$, read the entries of $T$ from $1$ to $n+m-r-k$ in increasing order. Entries appearing only in the first (resp. second) row correspond to the (1,0) (resp. (0,1)) step, and entries appearing in both rows correspond to the (1,1) step. Especially, by restricting the above mapping to $\mathrm{SYT}((n,m)/(r))$, we then obtain a bijection between $\mathrm{SYT}((n,m)/(r))$ and $\mathrm{Cat}(r;n,m)$.

The above bijection indicates that we can deal with enumerative problems of lattice paths via Young tableaux, and vice versa. Here we show an example by giving a new proof of Theorem \ref{thm:catlan} for the case $E>N$. Before that, we need a result of Chen and Stanley, which will also be used in Section 4.

A \emph{reverse tableau} of shape $\mu$ is an array of positive integers of shape $\mu$ that is weakly decreasing in rows and strictly decreasing in columns. Let $\mathrm{RT}(\mu,n)$ denote the set of all reverse tableaux of shape $\mu$ whose entries belong to $\{1,2,\dots,n\}$. Figure \ref{fig:reversetableau} shows all tableaux of $\mathrm{RT}((2,2),3)$.
\begin{figure}[ht]
\begin{minipage}{0.16\linewidth}
\centering
\begin{tikzpicture}[scale=0.6]
\foreach \x in {1,2}
 \foreach \y in {1,2}
 {
  \draw (\x,\y) +(-.5,-.5) rectangle ++(.5,.5);
 }

 \draw (1,1) node{2};
 \draw (1,2) node{3};
 \draw (2,1) node{2};
 \draw (2,2) node{3};
\end{tikzpicture}
\end{minipage}
\begin{minipage}{0.16\linewidth}
\centering
\begin{tikzpicture}[scale=0.6]
\foreach \x in {1,2}
 \foreach \y in {1,2}
 {
  \draw (\x,\y) +(-.5,-.5) rectangle ++(.5,.5);
 }

 \draw (1,1) node{2};
 \draw (1,2) node{3};
 \draw (2,1) node{1};
 \draw (2,2) node{3};
\end{tikzpicture}
\end{minipage}
\begin{minipage}{0.16\linewidth}
\centering
\begin{tikzpicture}[scale=0.6]
\foreach \x in {1,2}
 \foreach \y in {1,2}
 {
  \draw (\x,\y) +(-.5,-.5) rectangle ++(.5,.5);
 }

 \draw (1,1) node{1};
 \draw (1,2) node{3};
 \draw (2,1) node{1};
 \draw (2,2) node{3};
\end{tikzpicture}
\end{minipage}
\begin{minipage}{0.16\linewidth}
\centering
\begin{tikzpicture}[scale=0.6]
\foreach \x in {1,2}
 \foreach \y in {1,2}
 {
  \draw (\x,\y) +(-.5,-.5) rectangle ++(.5,.5);
 }

 \draw (1,1) node{2};
 \draw (1,2) node{3};
 \draw (2,1) node{1};
 \draw (2,2) node{2};
\end{tikzpicture}
\end{minipage}
\begin{minipage}{0.16\linewidth}
\centering
\begin{tikzpicture}[scale=0.6]
\foreach \x in {1,2}
 \foreach \y in {1,2}
 {
  \draw (\x,\y) +(-.5,-.5) rectangle ++(.5,.5);
 }

 \draw (1,1) node{1};
 \draw (1,2) node{3};
 \draw (2,1) node{1};
 \draw (2,2) node{2};
\end{tikzpicture}
\end{minipage}
\begin{minipage}{0.16\linewidth}
\centering
\begin{tikzpicture}[scale=0.6]
\foreach \x in {1,2}
 \foreach \y in {1,2}
 {
  \draw (\x,\y) +(-.5,-.5) rectangle ++(.5,.5);
 }

 \draw (1,1) node{1};
 \draw (1,2) node{2};
 \draw (2,1) node{1};
 \draw (2,2) node{2};
\end{tikzpicture}
\end{minipage}
\caption{All tableaux of $\mathrm{RT}((2,2),3)$.}
\label{fig:reversetableau} 
\end{figure}

Chen and Stanley gave a formula for
$$s_{\lambda/\mu}(1,q,q^2,\dots)/s_{\lambda}(1,q,q^2,\dots),$$
which is equivalent to the following result about the distribution of the major index over standard Young tableaux with skew shapes.
\begin{theorem}\cite[Thm. 1.2]{Chen-Stanley}
\label{thm:SHLF}
Let $\lambda$ and $\mu$ be partitions with $\lambda\supseteq \mu$ and $n\in \mathds{N}$ such that $l(\mu)\leq l(\lambda)\leq n$. Then we have
$$\frac{\sum_{T\in \mathrm{SYT}(\lambda/\mu)}q^{\mathrm{maj}(T)}}{\sum_{T\in \mathrm{SYT}(\lambda)}q^{\mathrm{maj}(T)}}=\frac{[|\lambda/\mu|]!}{[|\lambda|]!}\sum_{S\in \mathrm{RT}(\mu,n)}\prod_{u\in \mu}q^{1-S(u)}[\lambda_{S(u)}-c(u)],$$
where $c(u)=j-i$ for $u=(i,j)$, and $S(u)$ is the entry in the square $u$ in $S$.
\end{theorem}

Now we can give the proof of Theorem \ref{thm:catlan} for the case $E>N$ as follows. The proof for the case $E<N$ can also be obtained by considering the amajor index of standard Young tableaux.
\begin{proof}
Let $\varphi$ denote the bijection from $\mathrm{SYT}((n,m)/(r))$ to $\mathrm{Cat}(r;n,m)$ given above. It is not difficult to verify that $\varphi$ preserves the descent set for $E>N$. Thus we have
\begin{equation}
\label{equ:cat1}
\sum_{P\in \mathrm{Cat}(r;n,m)}q^{\mathrm{maj}(P)}=\sum_{T\in \mathrm{SYT}((n,m)/(r))}q^{\mathrm{maj}(T)}.
\end{equation}
Let $\lambda$ denote the partition $(n,m)$. Then we obtain from Theorem \ref{thm:SHLF} that
\begin{equation}
\label{equ:cat2}
\begin{aligned}
&\frac{\sum_{T\in \mathrm{SYT}(\lambda/(r))}q^{\mathrm{maj}(T)}}{\sum_{T\in \mathrm{SYT}(\lambda)}q^{\mathrm{maj}(T)}}\\
&= \frac{[n+m-r]!}{[n+m]!}\sum_{S\in \mathrm{RT}((r),2)}\prod_{u\in (r)}q^{1-S(u)}[\lambda_{S(u)}-c(u)]\\
&= \frac{[n+m-r]!}{[n+m]!}\sum_{i=0}^r q^{-i}[m][m-1]\cdots [m-i+1][n-i]\cdots [n-r+1]\\
&= q^{-m}\frac{[n+m-r]![m]![n-m]!}{[n+m]![n-r]!}\sum_{i=0}^r q^{m-i}\qbc{n-i}{n-m}\\
&= q^{-m}\frac{[n+m-r]![m]![n-m]!}{[n+m]![n-r]!}\left(\qbc{n+1}{n-m+1}-\qbc{n-r}{n-m+1}\right),
\end{aligned}
\end{equation}
where the last equality is derived from Equation \eqref{equ:q-hockey}.

Equation \eqref{equ:HLF} implies that
\begin{equation}
\label{equ:cat3}
\sum_{T\in \mathrm{SYT}((n,m))}q^{\mathrm{maj}(T)}=\frac{q^m[n-m+1]}{[n+1]}\qbc{n+m}{n}.
\end{equation}
Combining Equation \eqref{equ:cat1}$\sim$\eqref{equ:cat3} together, we then obtain Theorem \ref{thm:catlan} for the case $E>N$.
\end{proof}

Let $M=\{a_1^{n_1},a_2^{n_2},\dots,a_m^{n_m}\}$ be a multiset containing $n_i$ copies of $a_i$, where $a_1<a_2<\cdots <a_m$ and $n=n_1+n_2+\cdots +n_m$. We denote by $\sigma_M$ the set of all permutations of $M$. Let $\{\sigma_1,\sigma_2,\dots,\sigma_k\}$ be a collection of \emph{complementary} permutations of subsets of $M$, i.e., they are disjoint as subsets and their union equals to $M$. A permutation $\sigma\in \sigma_M$ is called a \emph{shuffle} of $\{\sigma_1,\sigma_2,\dots,\sigma_k\}$ if it contains $\sigma_i\,(1\leq i\leq k)$ as a subword. We denote by $\mathcal{F}(\sigma_1,\sigma_2,\dots,\sigma_k)$ the set of all shuffles of $\{\sigma_1,\sigma_2,\dots,\sigma_k\}$.

For the distribution of the major index over shuffles of permutations of normal sets, Garsia and Gessel gave the following remarkable result, which extended a classical result of MacMahon \cite{MacMahon} and Foata \cite{Foata}.
\begin{theorem}\cite[Thm. 3.1]{Gessel}
\label{lem:gessel}
Let $\{\sigma_1,\sigma_2,\dots,\sigma_k\}$ be a collection of complementary permutations of subsets of $\{1,2,\dots,n\}$. Then we have
\begin{equation*}
\sum_{\sigma\in \mathcal{F}(\sigma_1,\sigma_2,\dots,\sigma_k)}q^{\mathrm{maj}(\sigma)}=\qbc{n}{\mu_1,\mu_2,\dots,\mu_k}q^{\mathrm{maj}(\sigma_1)+\cdots+\mathrm{maj}(\sigma_k)},
\end{equation*}
where $\mu_i$ is the cardinality of $\sigma_i$.
\end{theorem}

It is easy to see that the above result holds for complementary permutations of subsets of a multiset by standardizing the words.

\begin{example}
For $M=\{1^2,2^3,3^2,4\}$, let $\sigma_1=1313$ and $\sigma_2=2242$ be complementary permutations of subsets of $M$. Let $\sigma=12321423$ be a shuffle of $\sigma_1$ and $\sigma_2$. Let $f(\sigma)=13642857$ be the standardization of $\sigma$. Then $f(\sigma)$ is a shuffle of $\tilde{\sigma_1}=1627$ and $\tilde{\sigma_2}=3485$, and we have  $D(f(\sigma))=D(\sigma)=\{3,4,6\}$.
\end{example}

\begin{proof}[Proof of Theorem \ref{thm:schroder}]
It is obvious that each generalized Schr\"{o}der path of $\mathrm{Sch}_k(r;n,m)$ can be viewed as a shuffle of a generalized Catalan path of $\mathrm{Cat}(r;n-k,m-k)$ and $D^k$. Thus Theorem \ref{thm:schroder} is obtained by combining Theorem \ref{thm:catlan} and \ref{lem:gessel} together.
\end{proof}

\section{The major index and the amajor  index for \boldmath \texorpdfstring{$\mathrm{RInc}_k((n,m)/(r))$}{LG}}
For a generalized Schr\"{o}der path $P=p_1p_2\cdots p_{n}$ and a given linear ordering of $\{E,D,N\}$, let $w$ be a bijection from the set of steps of $P$ to $\{1,2,\dots,n\}$, such that
$$w(p_i)>w(p_j)\Leftrightarrow p_i>p_j,\, \mathrm{or}\, p_i=p_j=D\, \mathrm{and}\, i<j.$$
We then call the word $w(P)=w(p_1)w(p_2)\cdots w(p_n)$ a \emph{diagonal-reverse labelling} of $P$.
\begin{lemma}
\label{lem:labelledschro}
Let $r,n,m$ and $k$ be positive integers with $\mathrm{Sch}_k(r;n,m)\neq \emptyset$. For a given linear ordering of $\{E,D,N\}$ and $P\in \mathrm{Sch}_k(r;n,m)$, let $w(P)$ denote a diagonal-reverse labelling of $P$. If $E>D>N$, then we have
\begin{small}
\begin{equation*}
\label{equ:labelledschro1}
\begin{aligned}
&\sum_{P\in \mathrm{Sch}_k(r;n,m)}q^{\mathrm{maj}(w(P))}\\
&= q^{\frac{k(k-1)}{2}}\qbc{n+m-r-k}{k}\bigg\{\qbc{n+m-r-2k}{n-r-k}-\qbc{n+m-r-2k}{n-k+1}\bigg\};
\end{aligned}
\end{equation*}
\end{small}
if $E<D<N$, then we have
\begin{small}
\begin{equation*}
\label{equ:labelledschro2}
\begin{aligned}
&\sum_{P\in \mathrm{Sch}_k(r;n,m)}q^{\mathrm{maj}(w(P))}\\
&= q^{\frac{k(k-1)}{2}}\qbc{n+m-r-k}{k}\bigg\{\qbc{n+m-r-2k}{n-r-k}-q^{r+1}\qbc{n+m-r-2k}{n-k+1}\bigg\}.
\end{aligned}
\end{equation*}
\end{small}
\end{lemma}
\begin{proof}
We just give the proof for $E>D>N$, while the proof for $E<D<N$ is almost the same. Let $W=D_kD_{k-1}\cdots D_1$ be a word with $E<D_1<D_2<\cdots <D_k<N$. We denote by $\mathcal{F}(\mathrm{Cat}(r;n,m),W)$ the set of all shuffles of $W$ with lattice paths in $\mathrm{Cat}(r;n,m)$.

Given $P\in \mathrm{Sch}_{k}(r;n,m)$, let $\psi(P)$ denote the word obtained from $P$ by replacing the $i-$th $D$ with $D_{k-i+1}$ for $1\leq i\leq k$. It is not difficult to see that $\psi$ gives a bijection from $\mathrm{Sch}_{k}(r;n,m)$ to $\mathcal{F}(\mathrm{Cat}(r,n-k,m-k),W)$. Moreover, the bijection $\psi$ satisfies $D(w(P))=D(\psi(P))$. Thus we obtain from Theorem \ref{thm:catlan} and \ref{lem:gessel} that
\begin{equation*}
\begin{aligned}
&\sum_{P\in \mathrm{Sch}_k(r;n,m)}q^{\mathrm{maj}(w(P))}\\
&= \sum_{\psi(P)\in \mathcal{F}(\mathrm{Cat}(r,n-k,m-k),W)}q^{\mathrm{maj}(\psi(P))}\\
&= q^{\frac{k(k-1)}{2}}\qbc{n+m-r-k}{k}\bigg\{\qbc{n+m-r-2k}{n-r-k}-\qbc{n+m-r-2k}{n-k+1}\bigg\}.
\end{aligned}
\end{equation*}
\end{proof}

\begin{proof}[Proof of Theorem \ref{thm:RInc}]
Let $\varphi$ denote the bijection from $\mathrm{RInc}_{k}((n,m)/(r))$ to $\mathrm{Sch}_k(r;n,m)$ given in Section 2. For a given linear ordering of $\{E,D,N\}$ and $P\in \mathrm{Sch}_k(r;n,m)$, let $w(P)$ denote a diagonal-reverse labelling of $P$. Then for any $T\in \mathrm{RInc}_k((n,m)/(r))$, we have
$$\mathrm{maj}(w(\varphi(T)))=\mathrm{maj}(T),\,\mathrm{if}\, E>D>N,$$
and
$$\mathrm{maj}(w(\varphi(T)))=\mathrm{amaj}(T),\,\mathrm{if}\, E<D<N.$$
In fact, if $\varphi(T)=p_1p_2\cdots p_{n+m-r-k}$, then for $E>D>N$, we have $i$ a descent of $w(\varphi(T))$ if and only if the pair $(p_i, p_{i+1})$ equals to $(D,N)$, $(E,N)$, $(E,D)$ or $(D,D)$. In all cases, $i$ is a descent of $T$. The discussion for $E<D<N$ is similar. Thus Theorem \ref{thm:RInc} is derived from Lemma \ref{lem:labelledschro}.
\end{proof}

Note that equations of Theorem \ref{thm:schroder} differ from equations of Theorem \ref{thm:RInc} by the factor $q^{k(k-1)/2}$. The above proof shows that the factor $q^{k(k-1)/2}$ arises from diagonal-reverse labelling of Schr\"{o}der paths, and thus gives an explanation for the question of Du, Fan and Zhao.

\section{The major index for \boldmath \texorpdfstring{$\mathrm{Inc}_k((n,m)/(r))$}{LG}}
We firstly extend a bijection of Pechenik \cite[Prop. 2.1]{Pechenik} as follows.
\begin{lemma}
\label{lem:InctoSYT}
There is an explicit bijection between $\mathrm{Inc}_k((n,m)/(r))$ and the union of $\mathrm{SYT}((n-k,m-k,1^k)/(r))$ and $\mathrm{SYT}((n-k,m-k+1,1^{k-1})/(r))$ that preserves the major index.
\end{lemma}
\begin{proof}
Given $T\in \mathrm{Inc}_k((n,m)/(r))$, let $A$ be the set of numbers that appear twice in $T$. Let $B$ be the set of numbers that appear in the second row immediately right of an element of $A$. Then $\chi(T)$ is a Young tableau produced by the following algorithm. We firstly delete all elements of $A$ from the first row of $T$ and all elements of $B$ from the second row, and obtain $\chi(T)$ by appending $B$ to the first column. Then $\chi(T)\in \mathrm{SYT}((n-k,m-k,1^k))$ if $T(2,m)$ appears only in the second row, and $\chi(T)\in \mathrm{SYT}((n-k,m-k+1,1^{k-1}))$ otherwise. See Figure \ref{fig:InctoSYT} for an example of $\chi$.
\begin{figure}[ht]
\centering
\begin{minipage}{0.75\linewidth}
\begin{tikzpicture}[scale=0.6]
\draw (0,1.5) node{$T_1:$};
\foreach \x in {2,...,5}
 \foreach \y in {1}
 {
  \draw (\x,\y) +(-.5,-.5) rectangle ++(.5,.5);
 }
 \foreach \x in {4,...,6}
 \foreach \y in {2}
 {
  \draw (\x,\y) +(-.5,-.5) rectangle ++(.5,.5);
 }
 \draw (4,2) node{1};
 \draw[red] (5,2) node{2};
 \draw[red] (6,2) node{4};
 \draw[red] (2,1) node{2};
 \draw[blue] (3,1) node{3};
 \draw[red] (4,1) node{4};
 \draw[blue] (5,1) node{5};

\path[->] (7.5,1.5) edge node [above] {$\chi$} (9,1.5);
\draw (11,1.5) node{$\chi(T_1):$};
 \foreach \x in {13}
 \foreach \y in {1,2,3}
 {
  \draw (\x,\y) +(-.5,-.5) rectangle ++(.5,.5);
 }
 \draw (14,3) +(-.5,-.5) rectangle ++(.5,.5);
 \draw (15,4) +(-.5,-.5) rectangle ++(.5,.5);
 \draw (13,1) node{5};
 \draw (13,2) node{3};
 \draw (13,3) node{2};
 \draw (14,3) node{4};
 \draw (15,4) node{1};
\end{tikzpicture}
\end{minipage}\\
\vspace{6pt}
\begin{minipage}{0.75\linewidth}
\begin{tikzpicture}[scale=0.6]
\draw (0,1.5) node{$T_2:$};
\foreach \x in {2,...,5}
 \foreach \y in {1}
 {
  \draw (\x,\y) +(-.5,-.5) rectangle ++(.5,.5);
 }
 \foreach \x in {4,...,6}
 \foreach \y in {2}
 {
  \draw (\x,\y) +(-.5,-.5) rectangle ++(.5,.5);
 }
 \draw (4,2) node{1};
 \draw[red] (5,2) node{2};
 \draw[red] (6,2) node{5};
 \draw[red] (2,1) node{2};
 \draw[blue] (3,1) node{3};
 \draw (4,1) node{4};
 \draw[red] (5,1) node{5};

\path[->] (7.5,1.5) edge node [above] {$\chi$} (9,1.5);
\draw (11,1.5) node{$\chi(T_2):$};
 \foreach \x in {13}
 \foreach \y in {1,2}
 {
  \draw (\x,\y) +(-.5,-.5) rectangle ++(.5,.5);
 }
 \draw (14,2) +(-.5,-.5) rectangle ++(.5,.5);
 \draw (15,2) +(-.5,-.5) rectangle ++(.5,.5);
 \draw (15,3) +(-.5,-.5) rectangle ++(.5,.5);
 \draw (13,1) node{3};
 \draw (13,2) node{2};
 \draw (14,2) node{4};
 \draw (15,2) node{5};
 \draw (15,3) node{1};
\end{tikzpicture}
\end{minipage}
\caption{An example of $\chi$ with $T_1,T_2\in \mathrm{Inc}_2((5,4)/(2))$. The entries appearing twice are colored red, and the entries appearing in the second row immediately right of a red number are colored blue.}
\label{fig:InctoSYT} 
\end{figure}
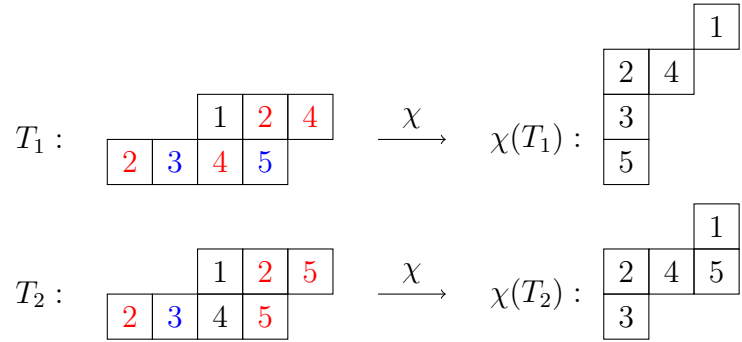

It is not hard to see that $\chi$ preserves the descent set and the major index. Thus we just need to show that $\chi$ is reversible.
For any $S$ in the union of $\mathrm{SYT}((n-k,m-k,1^k)/(r))$ and $\mathrm{SYT}((n-k,m-k+1,1^{k-1})/(r))$, let $B$ be the set of entries below the second row. By deleting entries below the second row and inserting $B$ into the second row of $S$ while maintaining increasingness, we obtain a tableau $T_1$. Let $A$ be the set of numbers immediately left of an element of $B$ in the second row of $T_1$. If $S\in\mathrm{SYT}((n-k,m-k,1^k)/(r))$, then $\chi^{-1}(S)$ is obtained by inserting $A$ into the first row of $T_1$ while maintaining increasingness. If $S\in \mathrm{SYT}((n-k,m-k+1,1^{k-1})/(r))$, then $\chi^{-1}(S)$ is obtained by inserting $A$ and $T_1(2,m)$ into the first row of $T_1$ while maintaining increasingness.
\end{proof}

The following result gives the distribution of the major index over standard Young tableaux with shape $(n,m,1^k)/(r)$.
\begin{lemma}
\label{lem:majofskewSYT}
For positive integers $r,n,m$ and $k$ with $\mathrm{SYT}((n,m,1^k)/(r))\neq \emptyset$, we have
\begin{equation*}
\label{equ:425}
\begin{aligned}
\sum_{T\in\mathrm{SYT}((n,m,1^k)/(r))}&q^{\mathrm{maj}(T)}=q^{\frac{k(k+1)}{2}}\qbc{n+m+k-r}{k}\\
&\bm{\cdot}\bigg(\frac{[m]}{[m+k]}\qbc{n+m-r}{m}-\frac{[n+1]}{[n+k+1]}\qbc{n+m-r}{n+1}\bigg).
\end{aligned}
\end{equation*}
\end{lemma}
\begin{proof}
Let $\lambda$ denote the partition $(n,m,1^k)$. By Theorem \ref{thm:SHLF}, we have
\begin{small}
\begin{equation}
\label{equ:421}
\begin{aligned}
\frac{\sum_{T\in \mathrm{SYT}(\lambda/(r))}q^{\mathrm{maj}(T)}}{\sum_{T\in \mathrm{SYT}(\lambda)}q^{\mathrm{maj}(T)}}& = \frac{[n+m+k-r]!}{[n+m+k]!}\\
&\bm{\cdot}\sum_{S\in \mathrm{RT}((r),k+2)}\prod_{u\in (r)}q^{1-S(u)}[\lambda_{S(u)}-c(u)].
\end{aligned}
\end{equation}
\end{small}

Given $S\in \mathrm{RT}((r),k+2)$, if $\lambda_{S(u_0)}<c(u_0)$ for some $u_0\in (r)$, then by \cite[Equation 2.19]{Chen-Stanley}, we have
$$\prod_{u\in (r)}q^{1-S(u)}[\lambda_{S(u)}-c(u)]=0.$$
Thus we can assume $S(1,i)\leq 2$ for $2\leq i \leq r$, and divide the sum of the right-hand side of Equation \eqref{equ:421} into two parts:
\begin{equation}
\label{equ:422}
\begin{aligned}
\sum_{S(1,1)=1}\prod_{u\in (r)}q^{1-S(u)}[\lambda_{S(u)}-c(u)]=\frac{[n]!}{[n-r]!},
\end{aligned}
\end{equation}
and
\begin{equation}
\label{equ:423}
\begin{aligned}
&\sum_{S(1,1)\geq 2}\prod_{u\in (r)}q^{1-S(u)}[\lambda^{1}_{S(u)}-c(u)]\\
&= \big(\frac{[m]}{q}+\sum_{i=2}^{k+1}\frac{1}{q^i}\big) \bm{\cdot}\bigg\{\frac{[m-1]![n-m]!}{q^{m-1}[n-r]!}\sum_{i=0}^{r-1}\qbc{n-i-1}{n-m}q^{m-i-1}\bigg\}\\
&= \frac{[m+k][m-1]![n-m]!}{q^{m+k}[n-r]!}\bigg(\qbc{n}{n-m+1}-\qbc{n-r}{n-m+1}\bigg),
\end{aligned}
\end{equation}
where the last identity is derived from Equation \eqref{equ:q-hockey}.

Equation \eqref{equ:HLF} implies that
\begin{small}
\begin{equation}
\label{equ:424}
\sum_{T\in \mathrm{SYT}((n,m,1^k))}q^{\mathrm{maj}(T)}=q^{m+\frac{k(k+3)}{2}}\frac{[n-m+1]}{[n+k+1]}\qbc{m+k-1}{k}\qbc{n+m+k}{n}.
\end{equation}
\end{small}
Combining Equation \eqref{equ:421} $\thicksim$ \eqref{equ:424} together, we then obtain Lemma \ref{lem:majofskewSYT}.
\end{proof}

Combining Lemma \ref{lem:InctoSYT} and \ref{lem:majofskewSYT}, we then obtain Theorem \ref{thm:Inc}. As an application, we can give another proof of Theorem \ref{thm:RInc} for the special case when $r=0$ by generalizing a bijection given in \cite[Thm. 6]{Du-Fan-Zhao} as follows.
\begin{lemma}
\label{lem:InctoRinc}
There is an explicit bijection between $\mathrm{RInc}_k((n,m))$ and the union of $\mathrm{Inc}_k((n,m))$ and $\mathrm{Inc}_{k-1}((n,m-1))$ that preserves the major index.
\end{lemma}
\begin{proof}
Given $T\in \mathrm{RInc}_k((n,m))$, we obtain the tableau $\rho(T)$ by the following algorithm. If $T\in \mathrm{Inc}_k((n,m))$, then $\rho(T)=T$. Otherwise, let $i$ be the minimal integer such that $T(1,i)=T(2,i)$.Then $\rho(T)$ is produced by firstly deleting $T(2,i)$ and then moving $T(2,j)$ one box to the left for $i<j\leq m$. For the later case, since $i$ is minimal, $T(2,i)-1$ appears only in the second row of $T$ or $i=1$, which implies that $T(2,i)-1$ is not a descent of $T$ and $\rho(T)$. Thus the operation $\rho$ preserves the descent set and the major index (see Fig. \ref{fig:RInctoInc}).
\begin{figure}[ht]
\centering
\begin{minipage}{0.8\linewidth}
\begin{tikzpicture}[scale=0.6]
\draw (0,1.5) node{$T:$};
\foreach \x in {2,...,5}
 \foreach \y in {1}
 {
  \draw (\x,\y) +(-.5,-.5) rectangle ++(.5,.5);
 }
 \foreach \x in {2,...,6}
 \foreach \y in {2}
 {
  \draw (\x,\y) +(-.5,-.5) rectangle ++(.5,.5);
 }
 \draw (2,2) node{1};
 \draw (3,2) node{2};
 \draw[red] (4,2) node{4};
 \draw (5,2) node{5};
 \draw (6,2) node{6};

 \draw (2,1) node{2};
 \draw (3,1) node{3};
 \draw[blue] (4,1) node{4};
 \draw (5,1) node{6};

\path[->] (7.5,1.5) edge node [above] {$\rho$} (9,1.5);
\draw (11,1.5) node{$\rho(T):$};
\foreach \x in {13,...,15}
 \foreach \y in {1}
 {
  \draw (\x,\y) +(-.5,-.5) rectangle ++(.5,.5);
 }
 \foreach \x in {13,...,17}
 \foreach \y in {2}
 {
  \draw (\x,\y) +(-.5,-.5) rectangle ++(.5,.5);
 }
 \draw (13,2) node{1};
 \draw (14,2) node{2};
 \draw (15,2) node{4};
 \draw (16,2) node{5};
 \draw (17,2) node{6};

 \draw (13,1) node{2};
 \draw (14,1) node{3};
 \draw (15,1) node{6};
\end{tikzpicture}
\end{minipage}
\caption{An example of $\rho$ with $T\in \mathrm{RInc}_3((5,4))\backslash\mathrm{Inc}_3((5,4))$.}
\label{fig:RInctoInc} 
\end{figure}

It is not hard to see that
$$\rho(T)\in \mathrm{Inc}_k((n,m))\cup \mathrm{Inc}_{k-1}((n,m-1)).$$
Therefore we just need to show that $\rho$ is reversible. If $T\in \mathrm{Inc}_k((n,m))$, we have $\rho^{-1}(T)=T$. If $T\in \mathrm{Inc}_{k-1}((n,m-1))$, let $i$ be the maximal integer such that $T(2,i)=T(1,i+1)-1$, where we assume $T(2,0)=0$. Then $\rho^{-1}(T)$ is obtained from $T$ by firstly moving $T(2,j)$ one box to the right for $i<j\leq m-1$, and then setting $T(2,i+1)=T(1,i+1)$.
\end{proof}

Applying the bijection in Lemma \ref{lem:InctoRinc} to Theorem \ref{thm:Inc}, we then obtain an alternate proof of Theorem \ref{thm:RInc} for the special case when $r=0$.
\begin{corollary}
For positive integers $n,m$ and $k$ with $\mathrm{RInc}_k((n,m))\neq \emptyset$, we have
\begin{small}
\begin{equation*}
\sum_{T\in \mathrm{RInc}_k((n,m))}q^{\mathrm{maj}(T)}=q^{m+\frac{k(k-3)}{2}}\frac{[n-m+1]}{[n-k+1]}\qbc{n+m-k}{k}\qbc{n+m-2k}{m-k}.
\end{equation*}
\end{small}
\end{corollary}

\emph{Jeu de taquin} (jdt) is a well-known transformation among skew Young tableaux. Readers can see \cite[Ch. 7, App. I]{Stanley} for the detailed definition. We show an example of the jdt transformation in Figure \ref{fig:jdt}, where bold numbers denote the entries moved during the transformation.

\begin{figure}[ht]
\centering
\begin{minipage}{1\linewidth}
\centering
\begin{tikzpicture}[scale=0.6]
\draw (0.5,2) node{$T=$};
\foreach \x in {2,3}
 \foreach \y in {1}
 {
  \draw (\x,\y) +(-.5,-.5) rectangle ++(.5,.5);
 }
 \foreach \x in {2,3,4}
 \foreach \y in {2}
 {
  \draw (\x,\y) +(-.5,-.5) rectangle ++(.5,.5);
 }
  \foreach \x in {3,4,5}
 \foreach \y in {3}
 {
  \draw (\x,\y) +(-.5,-.5) rectangle ++(.5,.5);
 }

 \draw (2,3) node{$a$};
 \draw (3,3) node{1};
 \draw (4,3) node{3};
 \draw (5,3) node{8};
 \draw (2,2) node{2};
 \draw (3,2) node{4};
 \draw (4,2) node{7};
 \draw (2,1) node{5};
 \draw (3,1) node{6};
 \draw (4,1) node{$b$};
\end{tikzpicture}
\end{minipage}\\
\vspace{15pt}
\begin{minipage}{0.49\linewidth}
\centering
\begin{tikzpicture}[scale=0.6]
\draw (0.5,2) node{$\mathrm{jdt}_a(T)=$};
\foreach \x in {3,4}
 \foreach \y in {1}
 {
  \draw (\x,\y) +(-.5,-.5) rectangle ++(.5,.5);
 }
 \foreach \x in {3,4}
 \foreach \y in {2}
 {
  \draw (\x,\y) +(-.5,-.5) rectangle ++(.5,.5);
 }
  \foreach \x in {3,4,5,6}
 \foreach \y in {3}
 {
  \draw (\x,\y) +(-.5,-.5) rectangle ++(.5,.5);
 }

 \draw (3,3) node{\textbf{1}};
 \draw (4,3) node{\textbf{3}};
 \draw (5,3) node{\textbf{7}};
 \draw (6,3) node{8};
 \draw (3,2) node{2};
 \draw (4,2) node{4};
 \draw (3,1) node{5};
 \draw (4,1) node{6};
\end{tikzpicture}
\end{minipage}
\begin{minipage}{0.49\linewidth}
\centering
\begin{tikzpicture}[scale=0.6]
\draw (0.5,2) node{$\mathrm{jdt}_b(T)=$};
\foreach \x in {3,4,5}
 \foreach \y in {1}
 {
  \draw (\x,\y) +(-.5,-.5) rectangle ++(.5,.5);
 }
 \foreach \x in {4,5}
 \foreach \y in {2}
 {
  \draw (\x,\y) +(-.5,-.5) rectangle ++(.5,.5);
 }
  \foreach \x in {4,5,6}
 \foreach \y in {3}
 {
  \draw (\x,\y) +(-.5,-.5) rectangle ++(.5,.5);
 }

 \draw (4,3) node{1};
 \draw (5,3) node{3};
 \draw (6,3) node{8};
 \draw (4,2) node{\textbf{2}};
 \draw (5,2) node{\textbf{4}};
 \draw (3,1) node{5};
 \draw (4,1) node{6};
 \draw (5,1) node{\textbf{7}};
\end{tikzpicture}
\end{minipage}
\caption{An example of the jdt transformation.}
\label{fig:jdt} 
\end{figure}

In \cite[Lem. 2.2]{Adin}, the jdt transformation is used to give a bijection between $\mathrm{SYT}((n-k+m,k)/(m))$ and the union $\bigsqcup_{d=k-m}^{\mathrm{min}\{k,n-k\}}\mathrm{SYT}(n-d,d)$. Here we can apply the jdt transformation to $\mathrm{SYT}((n-k+1,m-k+1,1^k)/(1^2))$ in a similar way, and obtain Theorem \ref{thm:RinctoSYT} as a corollary.
\begin{proof}[Proof of Theorem \ref{thm:RinctoSYT}]
We firstly give a bijection from $\mathrm{SYT}((n-k+1,m-k+1,1^k)/(1^2))$ to the union of $\mathrm{SYT}((n-k,m-k,1^k))$, $\mathrm{SYT}((n-k,m-k+1,1^{k-1}))$, $\mathrm{SYT}((n-k+1,m-k,1^{k-1}))$ and $\mathrm{SYT}((n-k+1,m-k+1,1^{k-2}))$. Given $T\in \mathrm{SYT}((n-k+1,m-k+1,1^k)/(1^2))$, let $a=(1,1)$ and $b=(2,1)$ denote the two boxes beside the northwest corner of $T$. Then the tableau $g(T)$ is defined to be $\mathrm{jdt}_a(\mathrm{jdt}_b(T))$. If $T_1=\mathrm{jdt}_b(T)\in\mathrm{SYT}((n-k+1,m-k,1^k)/(1))$, then we have $T_1(2,i)>T_1(1,i+1)$ for $1\leq i\leq m-k$, which implies that $g(T)\in \mathrm{SYT}((n-k,m-k,1^k))$. Otherwise, $g(T)$ belongs to the union of $\mathrm{SYT}((n-k,m-k+1,1^{k-1}))$, $\mathrm{SYT}((n-k+1,m-k,1^{k-1}))$ and $\mathrm{SYT}((n-k+1,m-k+1,1^{k-2}))$. See Figure \ref{fig:exampleofg} for an example of $g$.

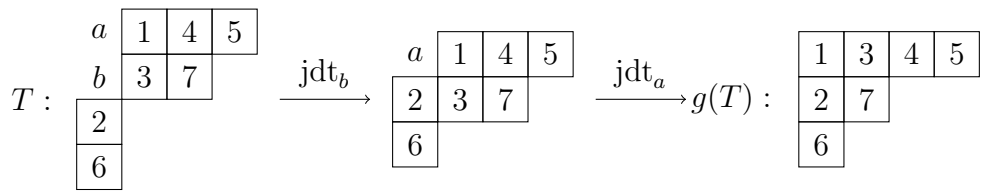
\begin{figure}[ht]
\centering
\begin{tikzpicture}[scale=0.6]
\draw (0.5,2) node{$T:$};
\draw (2,0.5) +(-.5,-.5) rectangle ++(.5,.5);
\draw (2,1.5) +(-.5,-.5) rectangle ++(.5,.5);
\foreach \x in {3,4}
 \foreach \y in {2.5}
 {
  \draw (\x,\y) +(-.5,-.5) rectangle ++(.5,.5);
 }
 \foreach \x in {3,4,5}
 \foreach \y in {3.5}
 {
  \draw (\x,\y) +(-.5,-.5) rectangle ++(.5,.5);
 }
 \draw (2,0.5) node{6};
 \draw (2,1.5) node{2};
 \draw (2,2.5) node{$b$};
 \draw (3,2.5) node{3};
 \draw (4,2.5) node{7};
 \draw (2,3.5) node{$a$};
 \draw (3,3.5) node{1};
 \draw (4,3.5) node{4};
 \draw (5,3.5) node{5};
\path[->] (6,2) edge node [above] {$\mathrm{jdt}_b$} (8,2);
\draw (9,1) +(-.5,-.5) rectangle ++(.5,.5);
\foreach \x in {9,10,11}
 \foreach \y in {2}
 {
  \draw (\x,\y) +(-.5,-.5) rectangle ++(.5,.5);
 }
 \foreach \x in {10,11,12}
 \foreach \y in {3}
 {
  \draw (\x,\y) +(-.5,-.5) rectangle ++(.5,.5);
 }
 \draw (9,1) node{6};
 \draw (9,2) node{2};
 \draw (10,2) node{3};
 \draw (11,2) node{7};
 \draw (9,3) node{$a$};
 \draw (10,3) node{1};
 \draw (11,3) node{4};
 \draw (12,3) node{5};
\path[->] (13,2) edge node [above] {$\mathrm{jdt}_a$} (15,2);
\draw (16,2) node{$g(T):$};
\draw (18,1) +(-.5,-.5) rectangle ++(.5,.5);
\foreach \x in {18,19}
 \foreach \y in {2}
 {
  \draw (\x,\y) +(-.5,-.5) rectangle ++(.5,.5);
 }
 \foreach \x in {18,...,21}
 \foreach \y in {3}
 {
  \draw (\x,\y) +(-.5,-.5) rectangle ++(.5,.5);
 }
 \draw (18,1) node{6};
 \draw (18,2) node{2};
 \draw (19,2) node{7};
 \draw (18,3) node{1};
 \draw (19,3) node{3};
 \draw (20,3) node{4};
 \draw (21,3) node{5};
\end{tikzpicture}
\caption{An example of $g$ with $T\in \mathrm{SYT}((4,3,1^2)/(1^2))$.}
\label{fig:exampleofg} 
\end{figure}

By the definition of jdt, it is easy to check that $g$ preserves the major index. We now construct the reverse of $g$ as follows. Given $S\in\mathrm{SYT}((n-k,m-k,1^k))$, let $a=(1,n-k+1)$ and $b=(2,m-k+1)$. Then we have $g^{-1}(S)=\mathrm{jdt}_b(\mathrm{jdt}_a(S))$. The construction of $g^{-1}$ for other cases is similar.

Combining Lemma \ref{lem:InctoSYT} and \ref{lem:InctoRinc} together, we know that $\chi\circ\rho$ is a bijection from $\mathrm{RInc}_k((n,m))$ to the union of $\mathrm{SYT}((n-k,m-k,1^k))$, $\mathrm{SYT}((n-k,m-k+1,1^{k-1}))$, $\mathrm{SYT}((n-k+1,m-k,1^{k-1}))$ and $\mathrm{SYT}((n-k+1,m-k+1,1^{k-2}))$. Thus $g^{-1}\circ\chi\circ\rho$ gives the required bijection from $\mathrm{RInc}_k((n,m))$ to $\mathrm{SYT}((n-k+1,m-k+1,1^k)/(1^2))$ that preserves the major index.
\end{proof}
\section{Acknowledgements}
We greatly appreciate the referees for their valuable comments and helpful suggestions adopted in this revised version.

\end{document}